\documentclass{article}

\usepackage[a4paper,top=2.54cm,bottom=2.54cm,left=3.17cm,right=3.17cm,%
            includehead,includefoot]{geometry}

\usepackage{amsmath,amssymb,amsfonts,amsthm}
\usepackage{graphicx}
\usepackage{subfigure}
\usepackage{float}
\usepackage{xcolor}
\usepackage[numbers,square,sort&compress]{natbib}
\usepackage{hyperref}
  \hypersetup{colorlinks,citecolor=blue,linkcolor=blue,breaklinks=true}
\usepackage{booktabs}
\usepackage{colortbl}
\usepackage{caption}
\usepackage{enumitem}
\usepackage{epstopdf}
\usepackage{algorithm}
\usepackage{algpseudocode}
\usepackage{array}
\usepackage{bbding}
\usepackage{fancyhdr} 
\usepackage{fancyvrb} 
\usepackage{longtable}
\usepackage{listings}
\usepackage{rotating,rotfloat} 
\usepackage{yhmath} 

\usepackage{fancyhdr}
\allowdisplaybreaks

\newtheorem{theorem}{Theorem}[section]
\newtheorem{lemma}{Lemma}[section]
\newtheorem{corollary}[theorem]{Corollary}
\newtheorem{proposition}[theorem]{Proposition}

\newtheorem{example}{Example}

\newcommand{\bbm}{\begin{bmatrix}}
\newcommand{\ebm}{\end{bmatrix}}

\begin{document}

\title{\uppercase{Fractal Structure of Parametric Cantor Sets With a Common Point}}

\author{
  Xinyi Meng
    \thanks{Address of the first author College of Mathematics and Statistics, Chongqing University, Chongqing 401331, China. }
}

\date{}

\maketitle

\begin{abstract}
For $\lambda>0$, let $E_{\lambda}$ be the self-similar set generated by the iterated function system (IFS) $\left \{ \frac{x}{3}, \frac{x+\lambda}{3} \right \}$. In this paper we study the structure of parameters $\lambda$ in which $E_\lambda$ contains a common point.  $E_{\lambda}$. More precisely, for a given point $x>0$ we consider the topology of the parameter set $\Lambda \left ( x \right ) =\left \{ \lambda >0:x\in E_{\lambda } \right \}$. We show that $\Lambda \left ( x \right )$ is a Lebesgue null set contains neither interior points nor isolated points, and the Hausdorff dimension of $\Lambda \left ( x \right ) $ is $ \log 2/ \log 3 $. Furthermore, we consider the set $\Lambda_{\mathrm {not}}(x)$ which consists of all parameters $\lambda$ that the digit frequency of $x$ in base $\lambda$ does not exist. We also consider the set $\Lambda_p(x)$ consisting of all $\lambda$ in which the digit frequency of $2$ in the base $\lambda$ expansion of $x$ is $p$. We show that the Hausdorff dimension of $ \Lambda _{\mathrm {not}} \left( x \right) $ is $\log2 /\log 3$ and the lower bound Hausdorff dimension of $ \Lambda _{p} \left( x \right) $ is $-p\log_3 p-(1-p)\log_3(1-p)$.

\end{abstract}

\section{Introduction}

For $\lambda >0$ let $E_{\lambda }$ be the self-similar set generated by the iterated function system (IFS) $\Psi _{\lambda } :=\left \{ f_{0}\left ( x \right )=\frac{x}{3}, f_{1}\left (x\right )=\frac{x+\lambda}{3} \right \} $. Then $E_{\lambda }$ is the unique nonempty compact set satisfying (cf.\cite{K14}) 
\begin{align*}
	E_{\lambda } = \left \{ \sum_{i=1}^{\infty } \frac{d_{i}}{3^{i}} : d_{i}\in \left \{ 0,\lambda \right \} ,i\in \mathbb{N}  \right \} .
\end{align*}
Clearly, 0 is a common point of $E_{\lambda }$ for all $\lambda >0 $. For other $x>0$ it is natural to ask how likely the self-similar sets $E_{\lambda }\left( \lambda >0\right) $ contain the common point $x$ ? 

Note that 
\begin{align}\label{E_lambda-Cantor}
	E_{\lambda }=\left \{ \frac{\lambda }{2}\sum_{i=1}^{\infty}\frac{d_{i} }{3^{i} }:d_{i}\in \left \{ 0,2 \right \} ,i\in \mathbb{N} \right \}=\frac{\lambda }{2}\mathrm{C_{1/3} },
\end{align}
where $\mathrm{C_{1/3}}$ is the middle third Cantor set. Clearly, the Hausdorff dimension of $\mathrm{C_{1/3}}$ is $\log 2/ \log 3$, then the Hausdorff dimension of $ E_{\lambda}$ is also $\log 2/ \log 3$ (cf.\cite{K85}).

Given $x>0$, let 
\begin{align*}
	\Lambda \left ( x \right ) =\left \{ \lambda >0:x\in E_{\lambda } \right \} .
\end{align*}
Then $\Lambda \left( x\right) $ consists of all $\lambda >0 $ such that $x$ is the common point of $E_{\lambda }$. It is interesting to study $\Lambda \left( x\right) $ for $x>0$.

Note that for $\lambda >0$, the IFS $\Psi_{\lambda}$ satisfies the open set condition (cf.\cite{CJ17}). So there exists a natural bijective map from the symbolic space $\Omega :=\left \{ 0,2 \right \} ^{\mathbb{N} } $ to the self-similar set $E_{\lambda }$, denoted by $\pi_{\lambda }$. Motivated by the works of \cite{ZJDW21}, there also exists a bijection from the parameter set $\Lambda \left( x\right) $ to the symbolic space $\Omega $, denoted by $\Phi_{x}$. See more details in Section 2.

Another motivation to study $\Lambda \left( x\right) $ is from the works of Jiang, Kong and Li \cite{JDW21}, where they studied the self-similar set $K_{\lambda }$ generated by the IFS $\left \{\lambda \left ( x+i \right ) :i=0,\cdots ,m-1 \right \}$ for a given integer $m \ge 2$ and $\lambda \in \left( 0,1/m \right]$. They considered the topological and fractal properties of the set $\Lambda_{*} \left( x\right) :=\left \{ \lambda \in \left ( 0,1/m \right ):x\in K_{\lambda }   \right \}$, showed that for any $x\in \left( 0,1\right) $ the set $\Lambda_{*} \left( x\right)$ is a topological Cantor set with zero Lebesgue measure and full Hausdorff dimension. Morover, the minimal value of $\Lambda_{*} \left( x\right)$ is $\frac{x}{m-1+x} $ and the maximal value of $\Lambda_{*} \left( x\right)$ is $1/m$.

Actually, the unique beta expansions is an effective tool to solve this kind of problems. Some recent progress on these can be found in \cite{VP07,CCS12,PD21,DWFZJ20} and the references therein.

Recall that a set $F\subset \mathbb{R} $ is called a Cantor set if it is a non-empty compact set containing neither interior nor isolated points. Our first result considers the topology of $\Lambda \left( x\right) $.

\begin{theorem}\label{Topology-of-lambda(x)}
	For any $x>0$, the set $\Lambda \left ( x \right ) $ is a Lebesgue null set contains neither interior points nor isolated points with $\min\Lambda \left( x\right)  = 2x$ and $\sup\Lambda \left( x\right)  = \infty $.
\end{theorem}

In fact, Theorem \ref{Topology-of-lambda(x)} suggests that for any $q>2x$, the set $\Lambda \left( x\right) \cap \left[2x,q \right] $ can be obtained by successively removing a sequence of open intervals from the closed interval between $\min\Lambda(x)$ and $\max\Lambda(x)$. At the end of Section 2 we will introduce a geometrical construction of $\Lambda \left( x\right) $ (see Figure \ref{fig:1}).

\begin{figure}[!htp]
	\centering 
	\includegraphics[width=0.5\textwidth]{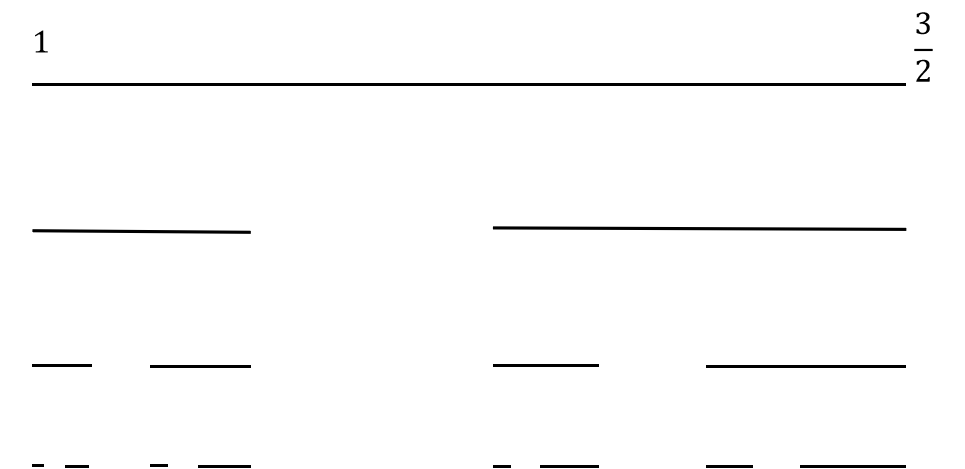}
	\caption{A geometrical construction of $\Lambda \left( x\right)\cap \left[ 0, 2 \right] $ with $x=1/2$.}
	\label{fig:1}
\end{figure}

Our next result states that the local dimension of $\Lambda \left( x\right) $.

\begin{theorem}\label{local-dimension}
	Let $x>0$. Then for any $\lambda \in \Lambda \left( x\right) $ and $\delta \in \left( 0,\lambda \right) $ we have
	\begin{align*}
		\dim_{\mathrm{H}} \left ( \Lambda \left ( x \right ) \cap \left ( \lambda - \delta ,\lambda +  \delta  \right )  \right ) = \dim_{\mathrm{H}} E_{\lambda }= \frac{\log 2}{\log 3} .
	\end{align*}
\end{theorem}

Inspired by the works of Albeverio, Pratsiovytyi and Torbin \cite{SMG05}, and the works of Huang, Kong \cite{YD23}, for a probability vector $\mathbf{p}:= \left(p,1-p \right) $ we define the sets
\begin{align*}
	\Lambda _{\mathrm{not}} \left ( x \right ) =\left \{ \lambda \in \Lambda \left ( x \right ):\underline{freq}_{i}\Phi _{x}\left ( \lambda  \right ) \ne \overline{freq} _{i}\Phi _{x}\left ( \lambda  \right ) ,i\in \left \{ 0,2 \right \}  \right \} ,
\end{align*}
and
\begin{align*}
	\Lambda _{p} \left ( x \right ) :=\left \{ \lambda \in \Lambda \left ( x \right ):freq_{2} \Phi _{x}\left ( \lambda \right ) = p,~freq_{0} \Phi _{x}\left ( \lambda  \right )= 1-p \right \} 
\end{align*}
where $freq_{i}\Phi_{x}\left( \lambda \right) \left(i=0,2 \right) $ denote the frequency of the digits 0 or 2 in $\Phi _{x}\left( \lambda \right) $.

Our next result is to consider the dimension of $\Lambda _{\mathrm{not} } \left ( x \right ) $.

\begin{theorem}\label{dimension-of-NOT}
	Let $x>0$, we have
	\begin{align*}
		\dim_{\mathrm{H}}\left( \Lambda _{\mathrm{not}}  \left ( x \right ) \right) = \frac{\log 2}{\log 3} .
	\end{align*}
\end{theorem}

The last result is the lower bound dimension of $\Lambda _{p} \left ( x \right ) $.

\begin{theorem}\label{dimension-of-p}
	Let $x>0$, for any $ p\in \left ( 0,1 \right ) $ we have
	\begin{align*}
		\dim_{\mathrm{H}}\left( \Lambda _{p} \left ( x \right )\right)  \ge \frac{\mathrm{h} \left ( p, 1-p \right ) }{\log 3} ,
	\end{align*}
	where $\mathrm{h} \left ( p, 1-p \right ) := -p\log p - \left ( 1-p \right ) \log \left ( 1-p \right ) $ is the entropy function of the probability vector $\mathbf{p}$.
\end{theorem}

The paper is organized in the following way. In Section 2 we show that $\Lambda \left( x\right) $ is a Lebesgue null set contains neither interior points nor isolated points, and prove Theorem \ref{Topology-of-lambda(x)}. In Section 3 we study the local dimension of $\Lambda \left( x\right) $, and prove Theorem \ref{local-dimension}. In Section 4 we construct subsets of $\Lambda _{\mathrm{not}}\left ( x \right )$ with the Hausdorff dimension arbitrarily close to $\log 2/ \log 3$, and prove Theorem \ref{dimension-of-NOT}. In Section 5 we constructed a measure supported on $\Lambda _{p} \left ( x \right ) $ and prove Theorem \ref{dimension-of-p}.

\section{Topological structure of \( \Lambda ( x )\) }

In this section we will investigate the topology of $\Lambda \left ( x \right )$ and prove Theorem \ref{Topology-of-lambda(x)}. We will give some properties of the mappings and then the topological properties of $\Lambda \left ( x \right )$ can be deduced.

First we recall some terminology from symbolic dynamics (cf.\cite{DB21}). Let $\Omega =\left \{ 0,2 \right \} ^{\mathbb{N}}$ be the set of all infinite sequences of 0s and 2s. For a word we mean a finite string of 0s and 2s. Let $\left\{ 0,2 \right\}^{*}$ be the set of all finite words over the alphabet {0, 2} together with the empty word $\epsilon$. For two words $\mathbf{a}= a_1 a_2\cdots a_k, \mathbf{b}=  b_1 b_2\cdots b_t $ in $\left\{ 0,2 \right\}^{*}$ we write $\mathbf{ab} = a_1 a_2\cdots a_k b_1 b_2\cdots b_t $ for their concatenation. Similarly, for a sequence $\mathbf{c}= c_1 c_2\cdots$ in $\Omega$ we write $\mathbf{ac}= a_1 a_2\cdots a_k c_1 c_2 \cdots$. For $j\in \mathbb{N} $ we denote by $\mathbf{a}^{j}=\mathbf{a}\cdots \mathbf{a}$ the j-fold concatenation of $\mathbf{a}$ with itself, and by $ \mathbf{a}^{\infty} = \mathbf{a} \mathbf{a} \cdots $ the period sequence with period block $\mathbf{a}$. Throughout the paper we use lexicographical order $\prec,\preceq,\succ$ and $\succeq$ between sequences and words. For example, for two sequences $\left ( a_{i} \right ) ,\left ( b_{i} \right ) \in \Omega $, we say $\left( a_{i} \right) \prec \left ( b_{i} \right ) $ if $a_{1} < b_{1}$ or there exists $n\in \mathbb{N} $ such that $a_{1}...a_{n-1} = b_{1}...b_{n-1}$ and $a_{n} < b_{n}$. For two words $\mathbf{a},\mathbf{b}$, we say $\mathbf{a} \prec \mathbf{b}$ if $\mathbf{a}0^{\infty} \prec \mathbf{b}0^{\infty} $.

Note that for $\lambda >0$, the IFS $\Psi_{\lambda}$ satisfies the open set condition. So there exists a natural bijective map from the symbolic space $\Omega $ to the self-similar set $E_{\lambda }$ defined by
\begin{align*}
	\pi_{\lambda}:\Omega \longrightarrow E_{\lambda } ;~~~~~~~~\left ( d_{i}\right ) \longmapsto \frac{\lambda}{2}\sum_{i=1}^{\infty } \frac{d_{i}}{3^{i}}.  
\end{align*}
The infinite sequence $\left( d_{i}\right) $ is called a (unique) coding of $\pi _{\lambda }\left( \left( d_{i}\right) \right) $ with respect to the IFS $\Psi_{\lambda}$. 

Now we fix $x>0$ and define a map from the parameter set $\Lambda \left ( x \right ) $ to the symbolic space $\Omega$ by
\begin{align*}
	\Phi _{x} :\Lambda \left ( x \right ) \longrightarrow \Omega;~~~~~~~~\lambda\longmapsto\left ( d_{i}  \right )~~ s.t.~~x=\frac{\lambda }{2} \sum_{i=1}^{\infty } \frac{d_{i} }{3^{i} } . 
\end{align*}

\begin{lemma}\label{bijection}
	Let $x>0$, $\Phi _{x}$ is a bijective mapping.
\end{lemma}

\begin{proof}
	Let $x>0$ and $\lambda_{1},\lambda_{2} \in \Lambda \left( x\right) $. Set $\left ( a_{i}  \right ) :=\Phi _{x} \left ( \lambda _{1}  \right ) $,$\left ( b_{i}  \right ) :=\Phi _{x} \left ( \lambda _{2}  \right ) $. If $\Phi _{x} \left ( \lambda _{1}  \right ) =\Phi _{x} \left ( \lambda _{2}  \right ) $ then
	\begin{align*}	
		\frac{\lambda _{1} }{2} \sum_{i=1}^{\infty } \frac{a_{i} }{3^{i} } =x=\frac{\lambda _{2} }{2} \sum_{i=1}^{\infty } \frac{b_{i} }{3^{i} } .
	\end{align*}
	So we have
	\begin{align*}	
		\lambda_{1}=\frac{2x}{ \sum_{i=1}^{\infty } \frac{a_{i} }{3^{i} }}=\frac{2x}{ \sum_{i=1}^{\infty } \frac{b_{i} }{3^{i} }}=\lambda_{2}.
	\end{align*}
	Thus $\Phi _{x}$ is injective.
	
	For any $\left( d_{i} \right) \in \Omega $, we have $ $
	\begin{align*}	
		x=\frac{\lambda}{2} \sum_{i=1}^{\infty } \frac{d_{i} }{3^{i} } >0 .
	\end{align*}
	For any $\left( d_{i} \right) \in \Omega $, $\lambda = x/ \sum_{i=1}^{\infty }\frac{d_{i}}{3^{i}}$, i.e., $x=\frac{\lambda}{2} \sum_{i=1}^{\infty } \frac{d_{i} }{3^{i} }$. Then $x\in E_{\lambda }$, i.e., $\lambda \in \Lambda \left( x\right) $.
	So $\Phi _{x}$ is surjective.
\end{proof}

\begin{lemma}\label{montonicity}
	For any $x>0$, $\Phi _{x} $ is strictly decreasing.
\end{lemma}

\begin{proof}
	For any $\lambda_{1},\lambda_{2} \in \Lambda \left ( x \right )$ and $\lambda _{1} < \lambda _{2} $, let $\left ( a_{i}  \right ) :=\Phi_{x} \left ( \lambda _{1}  \right )$ and $\left ( b_{i}  \right ):=\Phi_{x} \left ( \lambda _{2}  \right ) $. Then we have
	\begin{align*}	
		\frac{\lambda_{1} }{2} \sum_{n=1}^{\infty } \frac{a_{n} }{3^{n}} = x = \frac{\lambda_{2}}{2} \sum_{n=1}^{\infty } \frac{b_{n} }{3^{n}}.
	\end{align*}
	Since $\lambda_{1}<\lambda_{2}$,
	\begin{align}\label{con-}	
		\sum_{i=1}^{\infty } \frac{a_{i}-b_{i}}{3^{i} } >0.
	\end{align}	
	Note that $a_{i},b_{i}\in \left \{ 0,2 \right \} \left ( i\in\mathbb{N}  \right ) $. Thus, $a_{i}-b_{i}\in \left \{ -2,0,2 \right\}$ and there exists $i_{0}\in \mathbb{N} $ such that $a_{i_{0}}-b_{i_{0}}\ne 0$.
	
	Denoted by $i_{1} = \min \left \{ a_{i}-b_{i} : \left | a_{i} - b_{i}  \right | >0 \right \} $. Next, we will prove that $a_{i_{1}}-b_{i_{1}}>0$. Suppose $a_{i_{1}}-b_{i_{1}}<0$, that is, $a_{i_{1}}-b_{i_{1}}=-2$. Then we have	
	\begin{align*}	
		\sum_{i=1}^{\infty } \frac{a_{i}-b_{i}}{3^{i} } \le \frac{-2}{3^{i_{1} } }+\sum_{j=1}^{\infty }\frac{2}{3^{i_{1}+j} }=\frac{-1}{3^{i_{1}} }<0.
	\end{align*}	
	This is in contradiction with (\ref{con-}). Therefore, $a_{i_{1}}-b_{i_{1}}>0$. By the definition of $i_{1}$, we conclude that $\left ( a_{i}  \right ) \succ  \left ( b_{i}  \right )$.
\end{proof}

As a direct consequence of Lemma \ref{montonicity} we can determine the extreme values of $\Lambda \left( x\right) $.

\begin{lemma}\label{minimal value}
	For any $x>0$, we have $\min \Lambda \left( x\right) = 2x$.
\end{lemma}

\begin{proof}
	By Lemma \ref{bijection} and Lemma \ref{montonicity}, it follows that the smallest element $\lambda $ in $\Lambda \left( x\right) $ satisfies $\Phi _{x}\left( \lambda \right) = 2^{\infty } $. This gives that $x=\frac{\lambda }{2}$, then $\min \Lambda \left(x \right) = 2x $.
\end{proof}	

\begin{lemma}\label{max value}
	For any $x>0$, $\sup \Lambda \left ( x \right ) =\infty $.
\end{lemma}

\begin{proof}
	Since $x>0$, then for any $M>0$ let $\left( d_{i}\right) = 0^{n-1}20^{\infty } \in \Omega\left(n\in \mathbb{N} _{\ge 1} \right) $ such that $\sum_{i=1}^{\infty } \frac{d_{i}}{3^{i}} <\frac{2x}{M}$. And we can set $\lambda :=\frac{2x}{\sum_{i=1}^{\infty} \frac{d_{i}}{3^{i}} } $ then $\lambda \in \Lambda \left( x\right) $ and $\lambda >M$.
\end{proof}

\begin{proposition}\label{no interior points}
	Let $x>0$, the set $\Lambda \left ( x \right ) $ has no interior points.
\end{proposition}

\begin{proof}
	
	It suffices to prove that for any two points $\lambda_{1}, \lambda_{2} \in \Lambda \left ( x \right )$ there must exist $\lambda_{0} $ between $\lambda_{1}$ and $\lambda_{2}$ but not in $\Lambda \left( x\right) $. 
	
	Take $\lambda_{1},\lambda_{2} \in \Lambda \left ( x \right )$ with $\lambda_{1} < \lambda _{2}$. Denote $\left ( a_{i}\right ) :=\Phi_{x}\left ( \lambda _{1}\right ) $ and $\left ( b_{i}\right ) := \Phi_{x}\left ( \lambda _{2}\right )$. Then
	\begin{align*}
		\frac{\lambda_{1} }{2} \sum_{i=1}^{\infty } \frac{a_{i} }{3^{i} } = x = \frac{\lambda_{2} }{2} \sum_{i=1}^{\infty } \frac{b_{i} }{3^{i} }.
	\end{align*}
	By Lemma \ref{montonicity}, we have 
	\begin{align*}
		\left ( a_{i}\right ) =\Phi_{x}\left ( \lambda _{1}\right )\succ \Phi_{x}\left ( \lambda _{2}\right )=\left ( b_{i}\right ) .
	\end{align*}
	Since $\left( a_{i} \right) ,\left( b_{i} \right) \in \Omega $, then $\sum_{i=1}^{\infty } \frac{a_{i} }{3^{i} },\sum_{i=1}^{\infty } \frac{b_{i} }{3^{i} }\in \mathrm{C}_{1/3} $. Clearly, $\mathrm{C}_{1/3}$ contains no interior points, then there exists $\left( d_{i} \right) \in \left \{ 0,1,2 \right \} ^{\mathbb{N} } $ such that
	\begin{align*}
		\sum_{i=1}^{\infty } \frac{b_{i} }{3^{i} } < \sum_{i=1}^{\infty } \frac{d_{i} }{3^{i} } < \sum_{i=1}^{\infty } \frac{a_{i} }{3^{i} }~~\mbox{and}~~\left( d_{i} \right) \notin \Omega.
	\end{align*}
	Set $\lambda_{0} :=\frac{2x}{\sum_{i=1}^{\infty} \frac{d_{i}}{3^{i}} }$, then
	\begin{align}\label{between-lambda_0-and-lambda_2}
		\lambda_{1} < \lambda _{0} < \lambda _{2}.
	\end{align}	
	By Lemma \ref{bijection}, $\Phi _{x}$ is a bijective. Then
	\begin{align}\label{lambda_0-not-in}	
		\lambda _{0} = \frac{2x}{\sum_{i=1}^{\infty} \frac{d_{i}}{3^{i}} } \notin \Lambda\left( x\right) .
	\end{align}	
	Thus, by (\ref{between-lambda_0-and-lambda_2}) and (\ref{lambda_0-not-in}) we have $\lambda _{0} \in \left( \lambda _{1},\lambda_{2} \right)$ but $\lambda_{0}\notin \Lambda \left( x\right)$.	
\end{proof}

\begin{proposition}\label{no isolated points}
	Let $x>0$, the set $\Lambda \left ( x \right ) $ has no isolated points.
\end{proposition}

\begin{proof}
	For any $\lambda >0$, set $\left ( d_{i}  \right ) : = \Phi_{x} \left ( \lambda  \right ) $ and $\lambda _{n}$ satiesfies 
	\begin{align*}
		x=\frac{\lambda_{n}}{2}\left ( \sum_{i=1}^{n} \frac{d_{i}}{3^{i}} +\frac{2-d_{n+1}}{3^{n+1}}  \right ).
	\end{align*}
	Since $\left( d_{i}\right) \in \Omega$ then $2-d_{n+1}\in \left \{ 0,2 \right \}$. Thus $\lambda _{n}\in \Lambda \left( x\right) $, then we have 
	\begin{align}  \label{lambda_n-lambda}  
		\left | \lambda_{n} -\lambda \right | 
		&=\left | \frac{2x}{\sum_{i=1}^{n} \frac{d_{i}}{3^{i}} +\frac{2-d_{n+1}}{3^{n+1}}} - \frac{2x}{\sum_{i=1}^{\infty } \frac{d_{i}}{3^{i}} } \right | \nonumber \\
		&=2x\left | \frac{\sum_{i=n+1}^{\infty} \frac{d_{i}}{3^{i}}+\frac{d_{n+1}-2}{3^{n+1}}  }{\left ( \sum_{i=1}^{n} \frac{d_{i}}{3^{i}} +\frac{2-d_{n+1}}{3^{n+1}}\right ) \left ( \sum_{i=1}^{\infty } \frac{d_{i}}{3^{i}}  \right ) }  \right |\nonumber \\
		&\le 2x\left | \frac{\sum_{i=n+1}^{\infty} \frac{2}{3^{i}}}{\left ( \sum_{i=1}^{n} \frac{d_{i}}{3^{i}} +\frac{2-d_{n+1}}{3^{n+1}}\right ) \left ( \sum_{i=1}^{\infty } \frac{d_{i}}{3^{i}}  \right ) }  \right |\nonumber \\
		&=\frac{\lambda \lambda _{n}}{2x\cdot 3^{n}} .
	\end{align}   
	Now we prove $\lambda _{n}$ is bounded. Since
	\begin{align*}    
		\lambda _{n}=\frac{2x}{\sum_{i=1}^{n} \frac{d_{i}}{3^{i}} +\frac{2-d_{n+1}}{3^{n+1}}  } \le \frac{2x}{\sum_{i=1}^{n} \frac{d_{i}}{3^{i}} } .
	\end{align*}    
	and $x>0$, then
	\begin{align*}    
		\limsup_{n \to \infty} \lambda _{n}\le \limsup_{n \to \infty}\frac{2x}{\sum_{i=1}^{n} \frac{d_{i}}{3^{i}} }=\frac{2x}{\sum_{i=1}^{\infty } \frac{d_{i}}{3^{i}} } =\lambda.
	\end{align*}
	So there exists $N>0$ such that for $n>N$ we have $\lambda_{n}\le \lambda +1$. Then (\ref{lambda_n-lambda}) can be written as 
	\begin{align*}
		\left | \lambda_{n} -\lambda \right | \le \frac{\lambda \left ( \lambda +1 \right ) }{2x\cdot 3^{n}} .
	\end{align*}
	Letting $n\to \infty$ we have $\lambda _{n}\to \lambda $.
\end{proof}

Now, we proof the Theorem \ref{Topology-of-lambda(x)}.

\begin{proof}
	By Lemma \ref{minimal value} and Lemma \ref{max value}, Proposition \ref{no interior points} and Proposition \ref{no isolated points} we conclude the Theorem \ref{Topology-of-lambda(x)}.
\end{proof}

At the end of this section we state an example to describe the geometrical construction of $\Lambda \left( x\right) $. By Theorem \ref{Topology-of-lambda(x)} it follows that for any $q>2x$ the set $\Lambda\left( x\right) \cap \left[2x,q \right]$ contains neither interior points nor isolated points and $\Lambda\left( x\right) \cap \left[2x,q \right]$ is a compact set. Then $\Lambda \left( x\right) \cap \left[2x,q \right] $ can be obtained by successively removing a sequence of open intervals from the convex hull of $\Lambda \left( x\right) \cap \left[2x,q \right]$. 

\begin{example}\label{Example-Lambda(x)}
	Let $x=1/2, q=2$ then the convex hull of $\Lambda \left( x\right) \cap \left[2x,q \right] $ is $\left[1,3/2 \right]$. Furthermore, $\Phi_{x}\left(1\right) =2^{\infty}$ and $\Phi_{x}\left( 3/2 \right) =20^{\infty}$. Then by the properties of $\Phi_{x}$, in the first step we remove the open interval $I_{1}=\left( 1.1250, 1.2857 \right) \sim \left(220^{\infty}, 202^{\infty} \right) $ from the convex hull $\left[ 1,3/2 \right]$. And in the next step we remove two open intervals
	\begin{align*}
		I_{2} = \left( 1.0385 ,1.0800 \right) \sim \left(2220^{\infty}, 2202^{\infty} \right) ,\\
		I_{3} = \left( 1.3500, 1.4211\right) \sim \left(2020^{\infty}, 2002^{\infty} \right) .
	\end{align*}
	This procedure can be continued, and after finitely many steps we can get a good approximation of $\Lambda \left( x\right) $( see figure \ref{fig:1} ).
\end{example}

\section{Fractal properties of \(\Lambda \left ( x \right )\) }

In this section we investigate the local dimension of $\Lambda \left ( x \right )$. Given $x>0$ and $\lambda \in \Lambda \left( x\right) $, for any $\delta>0$ we firstly state the dimension of the set $\mathrm{C} _{1/3} \cap \left ( \frac{2x}{\lambda+\delta }  ,\frac{2x}{\lambda -\delta } \right )$.

\begin{lemma}\label{C=C}
	Given $\lambda \in \Lambda \left( x\right) $, and then for any $\delta \in \left( 0,\lambda \right) $
	\begin{align*}
		\dim _{\mathrm{H} } \mathrm{C} _{1/3} \cap \left ( \frac{2x}{\lambda+\delta }  ,\frac{2x}{\lambda -\delta } \right )=\dim _{\mathrm{H} } \mathrm{C} _{1/3}=\frac{\log 2}{\log 3}.
	\end{align*}
\end{lemma}

\begin{proof}
	Since $\lambda \in \Lambda \left( x\right) $ then $x\in E_{\lambda}$. By (\ref{E_lambda-Cantor}) there exists $\left( d_{i} \right) \in \Omega$ such that 
	\begin{align*}
		x=\frac{\lambda}{2}\sum_{i=1}^{\infty }  \frac{d_{i}}{3^{i}} .
	\end{align*}
	$\textbf{Case I}$: If $\left( d_{i}\right)$ is end with $2^{\infty}$, then their exists a large $n\in \mathbb{N} $ such that 
	\begin{align}\label{subset-of-A}
		A:=\left [ \sum_{i=1}^{n } \frac{d_{i}}{3^{i} } , \sum_{i=1}^{n} \frac{d_{i}}{3^{i} } +\frac{1}{3^{n} } \right ] \subset \left ( \frac{2x}{\lambda +\delta } , \frac{2x}{\lambda -\delta } \right ) .
	\end{align}
	Then for any $y\in A \cap \mathrm{C} _{1/3}$ there exists $\left( c_{i}\right) \in \Omega$ such that 
	\begin{align*}
		y=\sum_{i=1}^{\infty } \frac{c_{i}}{3^{i} } ~~\mbox{and}~~c_{1}\cdots c_{n}=d_{1}\cdots d_{n}.
	\end{align*}
	Define a mapping $ \varphi_{1}:A\cap \mathrm{C_{1/3}} \to \mathrm{C_{1/3}} $ such that for any $y\in A\cap \mathrm{C_{1/3}} $ we have $\varphi_{1}\left ( y \right ) =\sum_{i=1}^{\infty } \frac{c_{n+i}}{3^{i}} $. Then for any $y_{1},y_{2}\in A\cap \mathrm{C_{1/3}}$, there exists $\left( a_{i}\right) ,\left( b_{i}\right) \in \Omega $ such that 
	\begin{align*}
		y_{1}=\sum_{i=1}^{\infty }\frac{a_{i}}{3^{i}},y_{2}=\sum_{i=1}^{\infty }\frac{b_{i}}{3^{i}}~~\mbox{and}~~a_{1}\cdots a_{n}=b_{1}\cdots b_{n}=d_{1}\cdots d_{n}.
	\end{align*}
	Thus
	\begin{align*}
		\left | \varphi_{1}\left ( y_{1} \right ) -\varphi_{1}\left ( y_{2} \right )\right | &=\left | \sum_{i=1}^{\infty }\frac{a_{i}}{3^{i}}-\sum_{i=1}^{\infty }\frac{b_{i}}{3^{i}} \right | =\sum_{i=1}^{\infty } \frac{\left | a_{n+i}-b_{n+i} \right | }{3^{i}} \\
		&=3^{n}\sum_{i=1}^{\infty } \frac{\left | a_{n+i}-b_{n+i} \right | }{3^{n+i}} =3^{n}\left | y_{1} -y_{2}\right | .
	\end{align*}    
	Then $\varphi_{1}$ is a linear mapping. So we have
	\begin{align*}
		\dim _{\mathrm{H} }\left ( A\cap \mathrm{C_{1/3}} \right ) =\dim _{\mathrm{H} }\mathrm{C_{1/3}}=\frac{\log 2}{\log 3} .
	\end{align*}
	By (\ref{subset-of-A})
	\begin{align*}
		\dim _{\mathrm{H} }\left ( \left ( \frac{2x}{\lambda+\delta } , \frac{2x}{\lambda-\delta }  \right ) \cap \mathrm{C_{1/3}} \right ) \ge \dim _{\mathrm{H} }\left ( A\cap \mathrm{C_{1/3}} \right ).
	\end{align*}
	Since $\left ( \frac{2x}{\lambda+\delta } , \frac{2x}{\lambda-\delta }  \right ) \cap \mathrm{C}_{1/3} \subset \mathrm{C}_{1/3}$, then
	\begin{align*}
		\dim _{\mathrm{H} }\left ( \left ( \frac{2x}{\lambda+\delta } , \frac{2x}{\lambda-\delta }  \right ) \cap \mathrm{C_{1/3}} \right )\le \dim _{\mathrm{H} }\mathrm{C_{1/3}}.
	\end{align*}
	Therefore
	\begin{align*}
		\dim _{\mathrm{H} }\left ( \left ( \frac{2x}{\lambda+\delta } , \frac{2x}{\lambda-\delta }  \right ) \cap \mathrm{C_{1/3}} \right )= \frac{\log 2}{\log 3} .
	\end{align*}
	$\textbf{Case II}$: If $\left( d_{i}\right)$ does not end with $2^{\infty}$, there exists $\left \{ n_{k} \right \} \left ( k\in \mathbb{N} \right ) $ such that for any $k$ we have $ d_{n_{k}}=0 $. Since
	\begin{align*}  
		d_{1}\cdots d_{n_{k}}0^{\infty} \preceq d_{1}d_{2}\cdots,
	\end{align*} 
	then there exists a large $k\in \mathbb{N}$ such that
	\begin{align*}
		B:=\left [ \sum_{i=1}^{n_{k}} \frac{d_{n_{k} } }{3^{i}} , \sum_{i=1}^{n_{k}} \frac{d_{n_{k} } }{3^{i}}+\frac{1}{3^{n_{k} } } \right ] \subset \left ( \frac{2x}{\lambda +\delta } ,\frac{2x}{\lambda -\delta }  \right ) .
	\end{align*}
	For any $y\in B\cap \mathrm{C_{1/3}}$ there exists $\left( c_{i}\right) \in \Omega$ such that 
	\begin{align*}
		y=\sum_{i=1}^{\infty } \frac{c_{i}}{3^{i} } ~~\mbox{and}~~c_{1}\cdots c_{n_{k}}=d_{1}\cdots d_{n_{k}}.
	\end{align*}    
	Define a mapping $\varphi _{2}:B \cap \mathrm{C_{1/3}}\to \mathrm{C_{1/3}} $ such that $\varphi_{2}\left ( y \right ) =\sum_{i=1}^{\infty } \frac{c_{n_{k}+i}}{3^{i}} $.
	Similarly to $\textbf{Case I}$, $\varphi _{2} $ is linear and then 
	\begin{align*}
		\dim _{\mathrm{H} }\left ( \left ( \frac{2x}{\lambda+\delta } , \frac{2x}{\lambda-\delta }  \right ) \cap \mathrm{C_{1/3}} \right )=\dim _{\mathrm{H} } \mathrm{C_{1/3}}=\frac{\log 2}{\log 3} .
	\end{align*} 
	We complete the proof.
\end{proof}

To give the local dimension of $\Lambda \left( x\right) $ we still need the following mapping. For any $\delta \in \left( 0,\lambda \right) $, define a mapping
\begin{align*}
	f:\mathrm{C_{1/3} } \cap \left ( \frac{2x}{\lambda +\delta} ,  \frac{2x}{\lambda -\delta} \right )\to \Lambda\left ( x \right ) \cap \left ( \lambda-\delta, \lambda+\delta \right ).
\end{align*}	

We will state some properties of $f$.

\begin{proposition}\label{f-is-bijection}
	$f$ is a bijective mapping.
\end{proposition}

\begin{proof}
	Take $x_{1},x_{2}\in \mathrm{C_{1/3} } \cap \left ( \frac{2x}{\lambda +\delta} ,  \frac{2x}{\lambda -\delta} \right ) $, then there exists $\left( a_{i}\right) ,\left( b_{i} \right) \in \Omega $ with
	\begin{align*}
		x_{1}=\sum_{i=1}^{\infty } \frac{a_{i}}{3^{i}} ,~~ x_{2}=\sum_{i=1}^{\infty } \frac{b_{i}}{3^{i}} .
	\end{align*}
	By the definition of $\Lambda \left( x \right) $ we obtain that
	\begin{align*}
		\Lambda \left ( x \right ) \cap \left ( \lambda -\delta , \lambda +\delta \right ) = \left \{ \lambda ^{*} \in \left ( \lambda -\delta , \lambda +\delta \right ):\lambda ^{*}={2x}/~{\sum_{i=1}^{\infty } \frac{d_{i}}{3^{i} } },~\left ( d_{i}\right ) \in \Omega \right \} .
	\end{align*}
	Therefore, 
	\begin{align*}
		f\left( x_{1} \right) =2x/~{\sum_{i=1}^{\infty } \frac{a_{i}}{3^{i} } },~~f\left( x_{2} \right) = 2x/~{\sum_{i=1}^{\infty } \frac{b_{i}}{3^{i} } }.
	\end{align*}    
	If $f\left( x_{1} \right) =f\left( x_{2} \right) $, then 
	\begin{align*}
		{2x}/~{\sum_{i=1}^{\infty } \frac{a_{i}}{3^{i} } } = {2x}/~{\sum_{i=1}^{\infty } \frac{b_{i}}{3^{i} } }.
	\end{align*}
	Hence, 
	\begin{align*}
		x_{1} =\sum_{i=1}^{\infty } \frac{a_{i}}{3^{i}} = \sum_{i=1}^{\infty } \frac{b_{i}}{3^{i}} = x_{2}.
	\end{align*}
	So $f$ is injective.
	
	For any $\lambda ^{*}\in \Lambda\left ( x \right ) \cap \left ( \lambda -\delta, \lambda + \delta \right )$. Since $\lambda ^{*}\in \Lambda\left ( x \right )$ there exists $\left( d_{i} \right) \in \Omega$ satisfying
	\begin{align*}
		\lambda ^{*}=2x/~\sum_{i=1}^{\infty } \frac{d_{i}}{3^{i}}.
	\end{align*}
	We can set $y=\sum_{i=1}^{\infty } \frac{d_{i}}{3^{i}}$ where $\left( d_{i}\right) \in \Omega $ and then 
	\begin{align*}
		f\left ( y \right ) =f\left ( \sum_{i=1}^{\infty } \frac{d_{i}}{3^{i}} \right ) =f\left ( \frac{2x}{\lambda ^{*} }  \right ) =\lambda ^{*},
	\end{align*} 
	which implies that $f$ is surjective.
\end{proof}

\begin{proposition}\label{the-Lip-of-f}
	For any $\delta \in \left( 0,\lambda \right) $, there exists constants $ 0 < c_{1} \le c_{2} < \infty $ such that for any $x_{1},x_{2} \in \mathrm{C_{1/3} } \cap \left ( \frac{2x}{\lambda +\delta} ,  \frac{2x}{\lambda -\delta} \right ) $ 
	\begin{align*}	
		c_{1}\left | x_{1} -x_{2} \right | \le \left | f\left ( x_{1} \right ) -f\left ( x_{2} \right )  \right | \le c_{2} \left | x_{1} -x_{2}\right |.
	\end{align*}	
\end{proposition}

\begin{proof}
	Take $x_{1},x_{2}  \in \mathrm{C_{1/3} } \cap \left ( \frac{2x}{\lambda +\delta} ,  \frac{2x}{\lambda -\delta} \right )  $. Then there exists sequences $\left( a_{i}\right) ,\left( b_{i}\right) \in \Omega$ satisfying 
	\begin{align*}	
		x_{1}=\sum_{i=1}^{\infty } \frac{a_{i}}{3^{i}} ,~x_{2}=\sum_{i=1}^{\infty } \frac{b_{i}}{3^{i}} .
	\end{align*}	
	So
	\begin{align*}
		f\left ( x_{1} \right )=\frac{2x}{x_{1}},~f\left ( x_{2} \right )=\frac{2x}{x_{2}}.
	\end{align*}		
	Then
	\begin{align}\label{f(x1)-f(x2)}
		\left | f\left ( x_{1} \right ) -f\left ( x_{2} \right )  \right | = \left | \frac{2x}{x_{1}} -\frac{2x}{x_{2}} \right | =\frac{2x}{x_{1}x_{2}} \left | x_{1} -x_{2}\right | .
	\end{align}
	Notice that 
	\begin{align}\label{x_1,x_2}
		0<\frac{2x}{\lambda +\delta } < x_{1} \le x_{2} < \frac{2x}{\lambda -\delta } <\infty .
	\end{align}	
	So the expression (\ref{f(x1)-f(x2)}) can be written as 
	\begin{align}\label{inequation}	
		\frac{\left ( \lambda -\delta \right )^{2}  }{2x} \left | x_{1}-x_{2} \right | \le 	\left | f\left ( x_{1} \right ) -f\left ( x_{2} \right )  \right |\le \frac{\left ( \lambda +\delta \right )^{2}  }{2x} \left | x_{1}-x_{2} \right | .
	\end{align}		
	Let
	\begin{align*}
		c_{1} := \frac{\left ( \lambda -\delta \right )^{2}  }{2x},~ c_{2} := \frac{\left ( \lambda +\delta \right )^{2}  }{2x}.
	\end{align*}	
	Thus, $ 0 < c_{1} \le c_{2} < \infty$ as desired.
\end{proof}

Now, we proof Theorem \ref{local-dimension}.

\begin{proof}
	By Property \ref{f-is-bijection} and Property \ref{the-Lip-of-f}, $f$ is a bi-Lipschitz mapping from $\mathrm{C_{1/3} } \cap \left ( \frac{2x}{\lambda +\delta} ,  \frac{2x}{\lambda -\delta} \right )$ to $ \Lambda\left ( x \right ) \cap \left ( \lambda-\delta, \lambda+\delta \right )$. Then by Lemma \ref{C=C} we have
	\begin{align*}
		\dim _{\mathrm{H} } \Lambda \left ( x \right ) \cap \left ( \lambda-\delta,\lambda+\delta \right )
		& = \dim _{\mathrm{H} } f\left ( \mathrm{C} _{1/3} \cap \left ( \frac{2x}{\lambda+\delta }  ,\frac{2x}{\lambda -\delta } \right )  \right ) \\
		& = \dim _{\mathrm{H} } \mathrm{C} _{1/3} \cap \left ( \frac{2x}{\lambda+\delta }  ,\frac{2x}{\lambda -\delta } \right ) = \frac{\log 2}{\log 3}
	\end{align*}
	as desired.
\end{proof}

At the end of this section, we give a corollary of Theorem \ref{local-dimension}.

\begin{corollary}\label{dimension-of-Lambda(x)}
	Let $ x>0 $, for any $\lambda \in \Lambda\left( x\right) $ we have 
	\begin{align*}
		\dim_{\mathrm{H}}\left ( \Lambda \left ( x \right ) \right ) =\dim_{\mathrm{H}} E_{\lambda }= \frac{\log 2}{\log 3} .
	\end{align*}  
\end{corollary}

\begin{proof}
	By Theorem \ref{Topology-of-lambda(x)}, $\min \left \{ \lambda>0:\lambda \in \Lambda \left ( x \right )  \right \} =2x>0$ and there exists $ \left \{ \lambda_{n} \right \} \in \Lambda\left ( x \right ) $ such that $ \lim_{n \to \infty} \lambda _{n}=+\infty$. So for any $\delta \in \left(0,\lambda_{n} \right) $ we have 
	\begin{align*}
		\Lambda \left ( x \right ) =\bigcup_{n\in \mathbb{N} }\left ( \Lambda \left ( x \right ) \cap \left( \lambda_{n}-\delta , \lambda_{n}+\delta \right) \right ) .
	\end{align*}
	By Theorem \ref{local-dimension}, we obtain
	\begin{align*}
		\dim _{\mathrm{H} }\left ( \Lambda\left ( x \right )  \right ) =\sup_{n\in \mathbb{N} } \dim _{\mathrm{H}}\left ( \Lambda\left ( x \right ) \cap \left( \lambda_{n}-\delta , \lambda_{n}+\delta \right) \right ) = \frac{\log2}{\log3}
	\end{align*}
	as desired.	
\end{proof}

\section{The Hausdorff dimension of \(\Lambda_{\mathrm{not}} \left ( x \right )\) }

In this section, we will consider the Hausdorff dimension of the set 
\begin{align}\label{definition-Lambda_NOT}
	\Lambda _{\mathrm{not}} \left ( x \right ) =\left \{ \lambda \in \Lambda \left ( x \right ):\underline{freq}_{i}\Phi _{x}\left ( \lambda  \right ) \ne \overline{freq} _{i}\Phi _{x}\left ( \lambda  \right ) ,i\in \left \{ 0,2 \right \}  \right \},
\end{align}
where $freq_{i}\Phi_{x}\left( \lambda \right) \left(i=0,2 \right) $ denote the frequency of the digits 0 or 2 in $\Phi _{x}\left( \lambda \right) $. The main idea is to construct subsets of $\Lambda _{\mathrm{not}}\left ( x \right )$ with the Hausdorff dimension arbitrarily close to $\log 2/ \log 3$. 

Recall that for any sequence $\left( d_{n}\right) \in \Omega:=\left \{ 0,2 \right \} ^{\mathbb{N} } $, we denote
\begin{align*}
	\underline{freq}_{0} \left ( \left ( d_{n} \right )  \right )=\liminf_{n \to \infty} \frac{\# \left \{ 1\le k \le n :i_{k}=0 \right \} }{n} 
\end{align*}
and
\begin{align*}
	\overline{freq}_{0}\left ( \left ( d_{n} \right )  \right )=\limsup_{n \to \infty} \frac{\# \left \{ 1\le k \le n :d_{k}=0 \right \} }{n} .
\end{align*}
If $\underline{freq}_{0} \left ( \left ( d_{n} \right )  \right ) = \overline{freq}_{0}\left ( \left ( d_{n} \right ) \right )$, we say the frequency of digit 0 in the sequence $\left( d_{n}\right) $ exists, denoted by ${freq}_{0}\left ( \left ( d_{n} \right )  \right ) $. Similarly, we can define the frequency of digit 2 in $\left( d_{n}\right) $.

Fix $x>0$. For any $\lambda_{*} \in \Lambda \left ( x \right ) $ let $\left ( x_{i} \right ) :=\Phi _{x} \left ( \lambda_{*}  \right ) \in \Omega $, then $x=\frac{\lambda _{*}}{2} \sum_{i=1}^{\infty }\frac{x_{i}}{3^{i}}>0 $. Since $\lambda_{*}\ge 2x >0$ snd $\sum_{i=1}^{\infty }\frac{x_{i}}{3^{i}}>0 $, so there exists $n\in \mathbb{N} $ such that $x_{1} \cdot \cdot \cdot x_{n} \ne 0^{n} $. Let
\begin{align}\label{gamma_n-and-beta_n}
	\gamma _{n}:= \frac{2x}{\sum_{i=1}^{n} \frac{x_{i}}{3^{i}} } > \frac{2x}{\sum_{i=1}^{n} \frac{x_{i}}{3^{i}} +\sum_{i=n+1}^{\infty }\frac{2}{3^{i}}} =: \beta _{n}.
\end{align}
Then
\begin{align*}
	\Phi _{x}\left ( \gamma  _{n} \right )=x_{1}\cdot \cdot \cdot x_{n}0^{\infty }\prec \left( x_{i} \right) \prec \Phi _{x}\left ( \beta _{n} \right )=x_{1}\cdot \cdot \cdot x_{n}2^{\infty }.
\end{align*}
By Lemma \ref{montonicity}, we have 
\begin{align*}
	0< \beta _{n} \le \lambda_{*} \le \gamma _{n} < \infty .
\end{align*}
Since ${\sum_{i=1}^{\infty } \frac{x_{i}}{3^{i}} } =\frac{2x}{\lambda_{*}}>0$, then
\begin{align}\label{limit-is-equal}
	\lim_{n \to \infty} \gamma _{n}=\lim_{n \to \infty}  \frac{2x}{\sum_{i=1}^{n} \frac{x_{i}}{3^{i}} } = \frac{2x}{\sum_{i=1}^{\infty } \frac{x_{i}}{3^{i}} } =\lambda_{*}.
\end{align}
Similarly, $\lim_{n \to \infty} \beta _{n} =\lambda_{*} $.

Now, for $q\in \mathbb{N} $ define the set $ {\textstyle \sum_{q,n}} \left ( x \right ) $ consists of all $\lambda \in \Lambda \left( x\right)\cap \left( \beta _{n},\gamma_{n} \right) $ such that 
\begin{align}\label{Phi _{x}}
	\Phi _{x} \left ( \lambda  \right ) &= x_{1} \cdots x_{n} d_{1} d_{2} \cdots \nonumber \\
	& = x_{1} \cdot \cdot \cdot x_{n} d_{r_{0}+1 } d_{r_{0}+2} \cdot \cdot \cdot d_{r_{0}+3q} 002 \nonumber \\
	& ~~~~~~~~~~~~~~~~~~~~~~ d_{r_{1}+1 } d_{r_{1}+2} \cdots d_{r_{1}+2 \cdot 3q} 000022 \nonumber \\
	& ~~~~~~~~~~~~~~~~~~~~~~ \cdots\\
	& ~~~~~~~~~~~~~~~~~~~~~~ d_{r_{m-1}+1 } d_{r_{m-1}+2} \cdots d_{r_{m-1}+2^{m-1} \cdot 3q} 0^{2^{m} } 2^{2^{m-1} }\nonumber \\
	& ~~~~~~~~~~~~~~~~~~~~~~ d_{r_{m}+1 } d_{r_{m}+2} \cdots d_{r_{m}+2^{m} \cdot 3q} 0^{2^{m+1} } 2^{2^{m}} \cdots, \nonumber
\end{align}
where 
\begin{align}\label{Probability-r_m}
	r_{m} := 3\left ( q+1 \right )  \left ( 1+2+\cdots + 2^{m-1} \right ) = 3\left ( q+1 \right )\left ( 2^{m} -1 \right ) ~\mbox{and}~ d_{r_{m} + kq }=2
\end{align}
for all $k\in \left \{ 1,3,\cdots ,2^{m} \cdot 3 \right \} $ and $m\ge 0$. Here we point out that $r_{0}=0$. 

In the following we will show that $ {\textstyle \sum_{q,n}} \left ( x \right ) $ is a subset of $\Lambda_{\mathrm{not} }\left( x\right) $.

\begin{lemma}\label{the subset of Lambda_not_(x)}
	For any $ q,n\in \mathbb{N} $, we have
	\begin{align*}
		{\textstyle \sum_{q,n}} \left ( x \right ) \subset \Lambda _{\mathrm{not}} \left ( x \right ).
	\end{align*}
\end{lemma}

\begin{proof}
	For any $\lambda \in {\textstyle \sum_{q,n}} \left ( x \right ) $, set $\Phi _{x} \left ( \lambda  \right ) = x_{1} \cdots x_{n } d_{1} d_{2} \cdots $. It suffices to prove $\lambda \in\Lambda _{\mathrm{not}} \left ( x \right ) $.
	
	For any $m,n \in \mathbb{N}$, let $M_{2}  \left( \left( d_{i}\right),m \right)  $ and $M_{0} \left( \left( d_{i}\right),m \right) $ be the number of 2s and 0s in
	\begin{align*}
		d_{r_{0} + 1} \cdots d_{r_{0} + 3q } d_{r_{1} + 1} \cdots d_{r_{1} + 2\cdot 3q } \cdots d_{r_{m-1} + 1} \cdots d_{r_{m-1} + 2^{m-1}  \cdot 3q }
	\end{align*}
	and let $N_{2} \left ( \left ( d_{i} \right ) ,n \right ) $ and $N_{0} \left ( \left ( d_{i} \right ) ,n \right ) $ be the number of 2s and 0s in $d_{1} d_{2} \cdots d_{n}  $. By (\ref{Probability-r_m}), we have
	\begin{align}\label{N_2}
		N_{2} \left ( \left ( d_{i} \right ) ,r_{m}  \right ) &= M_{2}  \left( \left( d_{i}\right),m \right)  +\left ( 1+2+\cdots + 2^{m-1} \right ) \nonumber \\
		&= M_{2}  \left( \left( d_{i}\right),m \right)  +\left( 2^{m} -1\right) 
	\end{align}
	and
	\begin{align}\label{N_0}
		N_{0} \left ( \left ( d_{i} \right ) ,r_{m}  \right ) &= M_{0} \left( \left( d_{i}\right),m \right)+\left ( 2 + 2^{2} +\cdots + 2^{m} \right ) \nonumber \\
		&= M_{0} \left( \left( d_{i}\right) ,m \right)+ 2\left ( 2^{m} -1 \right ).
	\end{align}
	Suppose
	\begin{align*}	
		\limsup_{m \to \infty} \frac{M_{2}  \left( \left( d_{i}\right),m \right)  }{2^{m} } >\liminf_{m \to \infty} \frac{M_{2}  \left( \left( d_{i}\right),m \right)  }{2^{m} } 
	\end{align*}	
	and
	\begin{align*}	
		\limsup_{m \to \infty} \frac{M_{0}  \left( \left( d_{i}\right),m \right)  }{2^{m} } >\liminf_{m \to \infty} \frac{M_{0}  \left( \left( d_{i}\right),m \right)  }{2^{m} }. 
	\end{align*}	
	Then by (\ref{N_2}) and (\ref{N_0}), we obtain
	\begin{align*}		
		\limsup_{m \to \infty} \frac{N_{2}  \left( \left( d_{i}\right),r_{m} \right)  }{2^{m} } >\liminf_{m \to \infty} \frac{N_{2}  \left( \left( d_{i}\right),r_{m} \right)  }{2^{m} } 
	\end{align*}		
	and
	\begin{align*}	
		\limsup_{m \to \infty} \frac{N_{0}  \left( \left( d_{i}\right),r_m \right)  }{2^{m} } >\liminf_{m \to \infty} \frac{N_{0}  \left( \left( d_{i}\right),r_m \right)  }{2^{m} }. 
	\end{align*}
	Then $\lambda \in \Lambda_\mathrm{not} \left( x\right) $.
	
	Without loss of generality, assume that $\limsup_{m \to \infty} \frac{M_{2}  \left( \left( d_{i}\right),m \right)  }{2^{m} } = \liminf_{m \to \infty} \frac{M_{2}  \left( \left( d_{i}\right),m \right)  }{2^{m} }=\alpha $. By (\ref{Phi _{x}}) we have
	\begin{align}\label{the limit of N_2/r_m}
		\lim_{m \to \infty} \frac{ N_{2} \left ( \left ( d_{i} \right ) ,r_{m}  \right )}{r_{m} } = \frac{1}{3\left ( q+1 \right ) } \left ( 1+\lim_{m \to \infty} \frac{ M_{2}  \left( \left( d_{i}\right),m \right) }{2^{m}} \right ) =\frac{1+\alpha }{3\left ( q+1 \right ) } .
	\end{align}
	Denote $\ell _{m} = r_{m} - 2^{m-1}$ then $\ell _{m} = 3\left ( q+1 \right ) \left ( 2^{m} -1 \right ) - 2^{m-1}$. By (\ref{Probability-r_m}) we obtain
	\begin{align*}
		N_{2} \left ( \left ( d_{i} \right ) ,\ell _{m} \right ) 
		&= M_{2}  \left( \left( d_{i}\right),m \right) +\left ( 1+2+\cdots +2^{m-2} \right ) \\
		&= M_{2}  \left( \left( d_{i}\right),m \right)+ 2^{m-1} - 1,
	\end{align*}
	and 
	\begin{align}\label{the limit of N_2/l_m}
		\lim_{m \to \infty} \frac{N_{2} \left ( \left ( d_{i} \right ) ,\ell _{m} \right )}{\ell _{m}} 
		&= \lim_{m \to \infty} \frac{ M_{m,2}  \left( \left( d_{i}\right) \right)  + 2^{m-1} - 1 }{3\left ( q+1 \right ) \left ( 2^{m} -1 \right ) - 2^{m-1}}  \nonumber\\
		&=\frac{1}{6q+5}\left [ 1+2\lim_{m \to \infty} \frac{ M_{m,2}  \left( \left( d_{i}\right) \right) }{2^{m} } \right ] =\frac{1+2\alpha }{6q+5} .
	\end{align}
	By the equations (\ref{the limit of N_2/r_m}) and (\ref{the limit of N_2/l_m}), we conclude
	\begin{align*}
		\frac{1+\alpha }{3\left ( q+1 \right ) } =\frac{1+2\alpha }{6q+5},
	\end{align*}
	which is impossible. Because of by (\ref{N_2}), we have
	\begin{align*}	
		\alpha =\lim_{m \to \infty} \frac{M_{2}\left ( \left ( d_{i} \right ) ,m \right ) }{2^{m}} \le \lim_{m \to \infty} \frac{3q\left ( 2^{m} -1 \right ) }{2^{m}} =\lim_{m \to \infty} 3q\left ( 1-\frac{1}{2^{m}} \right )\le 3q. 
	\end{align*}
	Thus
	\begin{align*}
		\frac{1+\alpha }{3\left ( q+1 \right ) } -\frac{1+2\alpha }{6q+5}=\frac{-\alpha +3q+2}{3\left ( q+1 \right )\left ( 6q+5 \right )  } \ge \frac{2}{3\left ( q+1 \right )\left ( 6q+5 \right ) } >0,
	\end{align*}	
	which implies the frequency of 2s in $ \left ( d_{i}  \right ) $ does not exist. Similarly, the frequency of 0s in $ \left ( d_{i}  \right ) $ also does not exist. Then by (\ref{definition-Lambda_NOT}), $\lambda \in \Lambda _{\mathrm{not}} \left ( x \right ) $.
\end{proof}

\begin{lemma}\label{Lip-Property-in-lambda}
	For $\delta >0$ and $\lambda \in \Lambda \left( x\right)  \cap \left( \beta_{n}, \gamma_{n} \right) $. There exist constants $c_1,c_2>0$ such that for any $\lambda_{1} ,\lambda_{2} \in {\textstyle \sum_{q,n}} \left ( x \right )$ we have 
	\begin{align*}
		c_1 \left | \lambda_{1} -\lambda_{2}  \right | \le \left | \pi _{\lambda } \left ( \Phi _{x} \left ( \lambda_{1}  \right )  \right ) -  \pi _{\lambda } \left ( \Phi _{x} \left ( \lambda_{2}  \right )  \right ) \right | \le c_2 \left | \lambda_{1} -\lambda_{2}  \right |. 
	\end{align*}
\end{lemma}

\begin{proof}
	By using Lemma \ref{bijection}, we obtain that $\Phi_{x}$ is a bijective mapping. Take $\lambda _{1} , \lambda _{2} \in {\textstyle \sum_{q,j}}\left (x\right ) $ with $\lambda _{1} < \lambda _{2} $. Let $\left ( a_{i}  \right ): = \Phi _{x} \left ( \lambda _{1}  \right )$ and $\left (b_{i}  \right ) := \Phi _{x} \left ( \lambda _{2}  \right )  $ then 
	\begin{align*}
		x=\frac{\lambda _1}{2} \sum_{i=1}^{\infty } \frac{a_i}{3^i} ,~~x=\frac{\lambda _2}{2} \sum_{i=1}^{\infty } \frac{b_i}{3^i} .
	\end{align*}
	So
	\begin{align}\label{pi_lambda_Phi_x}
		\left | \pi _{\lambda} \left ( \Phi_x\left ( \lambda_1 \right )  \right ) - \pi _{\lambda} \left ( \Phi_x\left ( \lambda_2 \right )  \right ) \right |
		&=\left | \pi_{\lambda } \left ( \left ( a_i \right )  \right ) - \pi_{\lambda } \left ( \left ( b_i \right )  \right )\right | \nonumber  \\
		&=\left | \frac{\lambda }{2} \sum_{i=1}^{\infty} \frac{a_i}{3^i} - \frac{\lambda }{2} \sum_{i=1}^{\infty} \frac{b_i}{3^i}\right | \nonumber \\
		&= \left | \frac{\lambda }{2} \cdot \frac{2x}{\lambda_1} - \frac{\lambda }{2} \cdot \frac{2x}{\lambda_2} \right | \nonumber \\
		&= \frac{\lambda x}{\lambda_1\lambda_2} \left | \lambda_1 -\lambda_2\right |.
	\end{align}
	By the definition of $ {\textstyle \sum_{q,n}} \left ( x \right )$ we conclude that
	\begin{align*}
		\beta_{n} < \lambda _1 < \lambda_2 < \gamma_{n}~~\mbox{and}~~\beta_{n} < \lambda < \gamma_{n}.
	\end{align*}
	Then the equation (\ref{pi_lambda_Phi_x}) can be written as 
	\begin{align*}
		\frac{\beta_{n} x}{\gamma _{n}^{2} }  \left | \lambda_1 -\lambda_2\right | \le \frac{\lambda x}{\lambda_1\lambda_2} \left | \lambda_1 -\lambda_2\right |\le \frac{\gamma_{n} x }{\beta _{n}^{2} } \left | \lambda_1 -\lambda_2\right |.
	\end{align*}
	That is,
	\begin{align*}
		\frac{\beta_{n} x}{\gamma _{n}^{2} }   \left | \lambda_1 -\lambda_2\right | \le \left | \pi _{\lambda} \left ( \Phi_x\left ( \lambda_1 \right )  \right ) - \pi _{\lambda} \left ( \Phi_x\left ( \lambda_2 \right )  \right ) \right | \le \frac{\gamma_{n} x }{\beta _{n}^{2} } \left | \lambda_1 -\lambda_2\right |.
	\end{align*}
	Set $c_1 = \frac{\beta_{n} x}{\gamma _{n}^{2} } ,c_2 = \frac{\gamma_{n} x }{\beta _{n}^{2} }$ with $c_1,c_2>0$ as desired.
\end{proof}

\begin{lemma}\label{the lower dimension of sigma_q,j_(x)}
	For any $q,n\in \mathbb{N} $, we have 
	\begin{align*}
		\dim_{\mathrm{H}}  {\textstyle \sum_{q,n}} \left ( x \right ) \ge \frac{\left ( q-1 \right ) \log 2 }{\left ( q+1 \right ) \log 3}. 
	\end{align*}
\end{lemma}

\begin{proof}
	Let
	\begin{align*}
		E_{m} \left ( q \right ) {=} \left ( \left \{ 0,2 \right \} ^{q-1} \times \left \{ 2 \right \}  \right )^{3\cdot 2^{m} } \times \left \{ 0^{2^{m+1} } 2^{2^{m} }  \right \} .
	\end{align*}
	By (\ref{Phi _{x}}), we have
	\begin{align}\label{pp}
		\dim_{\mathrm{H}} \pi _{\lambda } \left ( \Phi _{x} \left (  {\textstyle \sum_{q,n}} \left ( x \right ) \right )  \right ) = \dim_{\mathrm{H}} \pi _{\lambda } \left ( \prod_{m=0}^{\infty } E_{m} \left ( q \right ) \right ).
	\end{align}
	Notice that each word in $E_{m} \left ( q \right ) $ has length $3\left ( q+1 \right ) 2^{m} $ and the number of elements in $E_{m} \left ( q \right ) $ is $ 2^{3\cdot 2^{m}\left ( q-1 \right ) } $. Furthermore, $\prod_{m=0}^{+\infty } E_{m} \left ( q \right )$ is the set of infinite sequences by concatenating words from each $E_{m} \left ( q \right ) $. So $\pi _{\lambda} \left ( \prod_{m=0}^{+\infty } E_{m} \left ( q \right ) \right )$ is a homogeneous Moran  set (cf.\cite{DZJ97}) satisfying the strong seqaration condition. Hence,
	\begin{align*}
		\dim_{\mathrm{H}} \pi _{\lambda } \left ( \prod_{m=0}^{+\infty } E_{m} \left ( q \right ) \right )
		&\ge \liminf_{m \to +\infty} \frac{\log \prod_{\ell =0}^{m-1}  2^{3\cdot 2^{\ell }\left( q-1 \right)  } }{\left( \sum_{\ell =0}^{m-1} 3\left ( q+1 \right )2^{\ell } \right) \log 3} \\
		&=\liminf_{m \to +\infty} \frac{\left( \sum_{\ell =0}^{m-1}3\cdot 2^{\ell }\left(q-1 \right) \right) \log 2 }{\left( 3\left ( q+1 \right ) \sum_{\ell =0}^{m-1} 2^{\ell }\right) \log3 } \\
		&=\frac{\left ( q-1 \right )\log 2 }{\left ( q+1 \right )\log3 } .
	\end{align*}
	By Lemma \ref{Lip-Property-in-lambda} and (\ref{pp}),
	\begin{align*}
		\dim _{\mathrm{H} }\left ( \Sigma _{q,n} \left ( x \right )  \right ) =	\dim_{\mathrm{H}} \pi _{\lambda } \left ( \Phi _{x} \left (  {\textstyle \sum_{q,j}} \left ( x \right ) \right )  \right ) \ge \frac{\left ( q-1 \right ) \log 2 }{\left ( q+1 \right ) \log 3 }
	\end{align*}
	as desired.	
\end{proof}

Now, we proof Theorem \ref{dimension-of-NOT}.

\begin{proof}
	By Lemma \ref{the subset of Lambda_not_(x)} and Lemma \ref{the lower dimension of sigma_q,j_(x)}, we obtain 
	\begin{align*}
		\dim_{\mathrm{H}}\left(  \Lambda _{\mathrm{not}} \left( x\right) \right) 
		\ge \dim_{\mathrm{H}} {\textstyle \sum_{q,n}} \left ( x \right ) 
		\ge \frac{\left ( q-1 \right ) \log 2}{ \left ( q+1 \right ) \log 3 } .
	\end{align*}
	Since $q\in \mathbb{N} $ is arbitrary, then
	\begin{align*}
		\dim_{\mathrm{H}}\left( \Lambda _{\mathrm{not}}\left ( x \right )  \right) \ge \frac{\log 2}{ \log 3 } 
	\end{align*}
	as $q\to \infty $. By Corollary \ref{dimension-of-Lambda(x)},
	\begin{align*}
		\dim_{\mathrm{H} }\left ( \Lambda_{\mathrm{not} }\left ( x \right )  \right ) \le \dim_{\mathrm{H}}\left ( \Lambda \left ( x \right )  \right ) =\frac{\log 2}{\log 3}.
	\end{align*}
	Thus
	\begin{align*}		
		\dim_{\mathrm{H} }\left ( \Lambda_{\mathrm{not} }\left ( x \right )  \right ) =\frac{\log 2}{\log 3}
	\end{align*}
	as desired.
\end{proof}

\section{The Hausdorff dimension of \(\Lambda_{p} \left ( x \right )\) }

For any probability vector $\mathbf{p}=\left ( p,1-p \right )$, where $  p\in \left ( 0,1 \right ) $. Define
\begin{align*}
	\Lambda _{p} \left ( x \right ) =\left \{ \lambda \in \Lambda \left ( x \right ):freq_{2} \Phi _{x}\left ( \lambda \right ) = p ,freq_{0} \Phi _{x}\left ( \lambda  \right ) = 1-p \right \} .
\end{align*}
We will construct a measure on $\Lambda _{p} \left ( x \right ) $ and use the Billingsley's Lemma to state the Hausdorff dimension of $\Lambda_{p} \left ( x \right )$.

Given $x>0$ and $ \lambda^{*} \in \Lambda \left ( x \right ) $. Let $\left ( x_{i} \right ): =\Phi _{x} \left ( \lambda^{*} \right ) \in \Omega $, i.e. $ x=\frac{\lambda ^{*} }{2} \sum_{i=1}^{\infty }\frac{x_{i}}{3^{i} }  $. Since $x > 0$, let $n\in \mathbb{N} $ such that $x_{1} \cdots x_{n} \ne 0^{n} $. Denote
\begin{align*}
	\gamma _{n}:= \frac{2x}{\sum_{i=1}^{n} \frac{x_{i}}{3^{i}} } > \frac{2x}{\sum_{i=1}^{n} \frac{x_{i}}{3^{i}} +\sum_{i=n+1}^{\infty }\frac{2}{3^{i}}} =: \beta _{n}.
\end{align*}
Similarly, we can verify $ \lim_{n \to \infty} \gamma _{n} = \lim_{n \to \infty} \beta _{n} = \lambda^{*} $. For any $k,n\in \mathbb{N} $, let $\mathcal {F} _{k,n} \left( x, p \right) $ be the set of $\lambda$ satiesfying 

(i)~~~$\lambda \in \Lambda \left ( x \right ) \cap \left ( \beta_{n}, \gamma_{n} \right )$ such that $\Phi _{x} \left ( \lambda  \right ) = x_{1} x_{2} \cdots x_{n} d_{1} d_{2} \cdots $.

(ii)~~$d_{nk+1} \cdots d_{nk+k} \ne 0^{k} $ and ${freq}_{2}\left ( \left ( d_{i} \right )   \right ) = p $.
\\Then for any $n\in \mathbb{N} $ and $ k\ge 1 $, we have 
\begin{align}\label{the subset of Lambda_p_(x)}
	\mathcal{F} _{k,n} \left ( x, p \right ) \subset \Lambda _{p} \left ( x \right ) \cap \left ( \beta_{n} ,\gamma_{n} \right ) .  
\end{align}

In the following, we estimate the dimension of the set $\mathcal{F} _{k,n} \left ( x, p \right )$.

\begin{lemma}\label{dimension-equal}
	For any $k,n\in \mathbb{N} $ and $\lambda \in \Lambda\left( x\right) \cap \left( \beta _{n},\gamma_{n}\right) $, we have 
	\begin{align*}
		\dim_{\mathrm{H}} \mathcal{F} _{k,n} \left ( x,p \right )  = \dim_{\mathrm{H}} \pi _{\lambda } \left ( \Phi _{x} \left ( \mathcal{F} _{k,n} \left ( x,p \right )  \right )  \right ). 
	\end{align*}
\end{lemma}

\begin{proof}
	The same argument of \ref{Lip-Property-in-lambda}, there exist constants $c_{2}>c_{1}>0$ such that for any $\lambda _{1},\lambda_{2} \in \mathcal{F} _{k,n} \left ( x,p \right )$, we have
	\begin{align*}
		c_{1} \left | \lambda_{1} - \lambda_{2} \right | \le \left | \pi_{\lambda }\left ( \Phi_{x}\left ( \lambda_1 \right )  \right )-\pi_{\lambda }\left ( \Phi_{x}\left ( \lambda_2 \right )  \right )  \right | \le c_{2} \left | \lambda_{1} - \lambda_{2} \right |
	\end{align*}
	as desired.
\end{proof}

Next, we consider the lower bound dimension of $\pi _{\lambda } \left ( \Phi _{x} \left ( \mathcal{F} _{k,n} \left ( x, p \right )  \right )  \right ) $. For any $k,n\in \mathbb{N} $ define a measure $\widehat{\mu } _{k,n}$ on the tree
\begin{align*}
	\mathrm{T}_{k,n} \left ( x \right ) := \left \{ x_{1} \cdots x_{n} d_{1} d_{2} \cdots : ~d_{nk+1} \cdots d_{nk+k} \ne 0^{k}~\forall~n\ge 0 \right \}.
\end{align*}
And for any $k\in \mathbb{N} $, let $\theta _{k} $ satiesfy 

\begin{align}\label{def-of-theta_k}
	\frac{1-\theta _{k}}{1-\left ( \theta _{k} \right )^{k} } = p .
\end{align}

\begin{lemma}\label{limit-of-theta_k}
	$\lim_{k \to \infty} \theta _{k}=1-p.$
\end{lemma}

\begin{proof}
	For any $k\in \mathbb{N}_{\ge 2}$,
	\begin{align*}
		1+\theta _{k}+\theta _{k}^{2}+\cdots + \theta _{k}^{k-1}=\frac{1}{p}=1+\theta _{k+1}+\theta _{k+1}^{2}+\cdots + \theta _{k+1}^{k}=\frac{1}{p}. 
	\end{align*}	
	Then
	\begin{align*}
		\left ( \theta _{k+1}- \theta _{k} \right )+\left ( \theta _{k+1}^{2}- \theta _{k}^{2} \right )+ \cdots +\left ( \theta _{k+1}^{k-1}- \theta _{k}^{k-1} \right )+ \theta _{k+1}^{k}=0.
	\end{align*}
	Since
	\begin{align*}
		\theta _{k+1}^{i}- \theta _{k}^{i} =\left ( \theta _{k+1}- \theta _{k} \right ) \left ( \theta _{k+1}^{i-1}+\theta _{k+1}^{i-2} \theta _{k} +\cdots + \theta _{k+1} \theta _{k+1}^{i-2} +\theta _{k}^{i-1} \right ) 
	\end{align*}
	and for any $k\in \mathbb{N}_{\ge 2}$ we have $\theta_{k}>0$, then $\theta _{k+1}<\theta _{k}$.
	
	Let $k_{0}=\left \lfloor \frac{1}{p}  \right \rfloor +2$. And for any $k\in \mathbb{N}_{\ge k_{0}} $ denote $\theta _{k}$ by
	\begin{align}\label{equation-theta_k}
		1+\theta _{k}+\theta _{k}^{2}+\cdots + \theta _{k}^{k-1}=\frac{1}{p}.
	\end{align}	
	Then for any $k\in \mathbb{N}_{\ge k_{0}}$, $\theta_{k}>0$. Since $k_{0}=\left \lfloor \frac{1}{p}  \right \rfloor +2$, $\theta_{k_{0}}<1$. If not, we assume $\theta_{k_{0}}\ge 1$ then
	\begin{align*}
		1+\theta _{k_{0}} +\cdots + \theta _{k_{0}}^{k_{0}-1}\ge 1+k_{0}-1=k_{0}=\left \lfloor \frac{1}{p} \right \rfloor +2>\frac{1}{p},
	\end{align*}
	which is a contradiction. By the monotonic decreasing property of $\left \{ \theta_{k} \right \}$ , we obtain the limit of $\left \{ \theta_{k} \right \}$ exists, denoted by $\theta_{0}$. Then $\theta _{0}\le \theta_{k}<1$.
	Taking the limit of both sides of the equation (\ref{equation-theta_k}), we obtain
	\begin{align*}
		\frac{1}{p}= \sum_{i=0}^{\infty } \theta ^{i} = \frac{1}{1-\theta } .
	\end{align*}
	Thus $\theta =\lim_{k \to \infty} \theta _{k}=1-p.$	
\end{proof}

Let $\left ( p_{0}, p_{2} \right ) = \left ( \theta _{k} ,1-\theta _{k} \right ) \left( k\in \mathbb{N}_{\ge k_{0}} \right)  $ be a probability vector. For any cylinder on $\mathrm{T}_{k,n} \left ( x \right ) $, $\widehat{\mu } _{0} \left ( \left [ x_{1} \cdots x_{n}  \right ]  \right ) = 1 $. For any $ n\ge 1$, 

\begin{align}\label{the measure mu}
	\widehat{\mu } _{n} \left ( \left [ x_{1} \cdots x_{n} d_{1} \cdots d_{nk}  \right ]  \right ) =\prod_{i=0}^{n-1} \frac{p_{d_{ik+1} }p_{d_{ik+2} }\cdots p_{d_{ik+k} } }{1-p_{0}^{k}}.
\end{align}

By \cite{GB99}, there exists a unique probability measure $\widehat{\mu } =\widehat{\mu }_{k,n} $ on $ \mathrm{T}_{k,n} \left ( x \right ) $ such that for any $ n\ge 0 $, we have
\begin{align*}
	\widehat{\mu } \left ( \left [ x_{1} \cdots x_{n} d_{1} \cdots d_{nk}  \right ]  \right ) = \widehat{\mu } _{n} \left ( \left [ x_{1} \cdots x_{n} d_{1} \cdots d_{nk} \right ]  \right ).
\end{align*}

\begin{lemma}\label{full measure}
	$\widehat{\mu } \left(  \Phi _{x} \left ( \mathcal{F} _{k,n} \left ( x, p \right )  \right ) \right) =1.$
\end{lemma}

\begin{proof}
	Notice $\Phi _{x} \left ( \mathcal{F}_{k,n} \left ( x, p \right ) \right ) \subset \mathrm{T}_{k,n} \left ( x \right )$. Since $\widehat{\mu } \left ( \mathrm{T}_{k,n} \left ( x \right )  \right ) =1$, it suffices to show that for any $\left ( c_{i} \right ) \in \mathrm{T}_{k,n} \left ( x \right )$ we have $freq_{2}\left ( \left ( c_{i}  \right )  \right ) = p, \widehat{\mu }-a.e. $.
	
	Note that for any $\left( c_{i}\right) \in \mathrm{T}_{k,n} \left ( x \right )$, $c_{1} \cdots c_{n} = x_{1} \cdots x_{n}$. And for any $N \in \mathbb{N}$ on $\widehat{\mu }$, $c_{n+Nk+1} \cdots c_{n+Nk+k}$ are independent and identically distributed. Thus by Law of Large Numbers (c.f\cite{R19}), we conclude that for any $\left ( c_{i} \right ) \in \mathrm{T}_{k,n} \left ( x \right ), \widehat{\mu }-a.e.$, 
	\begin{align*}
		freq_{2} \left ( \left ( c_{i}  \right )  \right ) &=\lim_{N \to +\infty} \frac{\# \left \{ 1\le i\le N:c_{i} = 2 \right \}  }{N} \\
		& =\lim_{N \to +\infty} \frac{\# \left \{ n+1 \le i\le n+Nk :c_{i} = 2 \right \}  }{Nk} \\
		& =\sum_{d_{1}\cdots d_{k} \ne 0^{k} } \frac{\#\left \{ 1\le i\le k:d_{i} =2 \right \} }{k} \widehat{\mu } \left ( \left [ x_{1} \cdots x_{n} d_{1} \cdots d_{k} \right ]  \right )\\
		&=\sum_{l=1}^{k} \frac{l}{k} \begin{pmatrix}
			k  \\
			l
		\end{pmatrix}
		\frac{\left( \theta _{k}\right)^{k-l}  \left ( 1-\theta _{k}  \right ) ^{ l } }{1-\left ( \theta _{k} \right ) ^{k} } .
	\end{align*}
	Rearrange the above summation expression and by (\ref{def-of-theta_k}), we obtain 
	\begin{align*}
		freq_{2} \left ( \left ( c_{i}  \right )  \right ) &= \frac{1}{1-\left ( \theta _{k}  \right ) ^{k} } \sum_{l =1 }^{k} \frac{\left ( k-1 \right ) ! }{\left ( l -1 \right ) ! \left ( k-l  \right ) ! } \left ( \theta _{k} \right )  ^{k-l }  \left ( 1-\theta _{k}  \right ) ^{l} \\
		&= \frac{1-\theta _{k}}{1-\left ( \theta _{k} \right ) ^{k} } = p
	\end{align*}
	as desired.
\end{proof}

Now, we prove Theorem \ref{dimension-of-p}.

\begin{proof}
	By Lemma \ref{full measure}, the set $\pi _{\lambda } \left( \Phi _{x} \left ( \mathcal{F} _{k,n} \left ( x, p \right ) \right)  \right ) $ has a full measure $\mu := \widehat{\mu } \circ \pi _{\lambda }^{-1} $. For any $y=\pi _{\lambda } \left ( x_{1} \cdots x_{n} d_{1} d_{2} \cdots \right ) \in \pi _{\lambda } \left(  \Phi _{x} \left ( \mathcal{F} _{k,n} \left ( x, p \right ) \right)  \right ),\mu-a.e.$, 
	\begin{align*}
		\liminf_{r \to 0^{+} } \frac{\log \mu \left ( B\left ( y,r \right )  \right ) }{\log r} 
		& = \liminf_{N \to +\infty } \frac{\log \widehat{\mu } \left ( \left [ x_{1} x_{2} \cdots x_{n} d_{1} d_{2} \cdots d_{Nk} \right ]  \right )  }{\log 3^{-\left( n+Nk\right) }  } \\
		& = \liminf_{N \to \infty} \frac{\log \prod_{i=0}^{N-1} \left ( \frac{p_{d_{ik+1} }p_{d_{ik+2} }\cdots p_{d_{ik+k}}}{1-p_{0}^{k} }  \right ) }{\log 3^{-\left( n+Nk\right) } } \\
		& = \liminf_{N \to \infty} \frac{\log \prod_{i=1}^{Nk}p_{d_{i} } -\log \left ( 1-p_{0}^{k}  \right )^{N}   }{-\left ( n+Nk \right ) \log 3 }  \\
		& = \liminf_{N \to \infty} \frac{M_{2} \left ( Nk \right ) \log p_{2}+ M_{0} \left ( Nk \right ) \log p_{0} -N\log \left ( 1- p_{0}^{k}  \right ) }{ -Nk \log 3},
	\end{align*}
	where $M_{i} \left ( Nk \right ) \left( i=0,2\right) $ denote the number of digits $i\left ( i=0,2 \right ) $ in $d_{1} \cdots d_{Nk} $. By Lemma \ref{full measure} for any $y\in \pi _{\lambda } \left( \Phi _{x} \left ( \mathcal{F} _{k,n} \left ( x, p \right )  \right ) \right),\mu-a.e.$, we have 
	\begin{align*}
		\liminf_{r \to 0^{+} } \frac{\log \mu \left ( B\left ( y,r \right )  \right ) }{\log r} 
		=\frac{\frac{1-\theta _{k} }{1-\left ( \theta _{k}  \right )^{k}  }\log \left( 1-\theta _{k}\right) +\frac{\theta _{k}- \left ( \theta ^{k}  \right ) ^{k}  }{1-\left ( \theta ^{k}  \right ) ^{k} } \log \theta _{k} }{-\log 3 } -\frac{\log\left ( 1-\left ( \theta _{k}  \right )^{k}   \right ) }{-k\log 3 } .
	\end{align*}
	By (\ref{def-of-theta_k}) and Lemma \ref{limit-of-theta_k}, we have 
	\begin{align*}
		\lim_{k \to \infty} \liminf_{r \to 0^{+}} \frac{\log \mu \left ( B\left ( y,r \right )  \right ) }{\log r}=\frac{p \log p +\left ( 1-p \right ) \log \left ( 1-p \right ) }{ -\log 3 } =\frac{\mathrm{h}\left ( p, 1-p \right ) }{ \log 3 },
	\end{align*}
	where $\mathrm{h}\left ( p, 1-p \right )$ denotes the entropy function of $ \mathbf{p}=\left ( p,1-p \right )  $. By \cite{P61} for any $n\ge 1$, we obtain 
	\begin{align*}
		\dim_{\mathrm{H}} \pi _{\lambda } \left(  \Phi _{x} \left ( \mathcal{F} _{k,n} \left ( x, p \right )  \right )\right)  \ge \frac{\mathrm{h} \left ( p, 1-p \right ) }{ \log 3 } . 
	\end{align*}
	So by (\ref{the subset of Lambda_p_(x)}) and Lemma \ref{dimension-equal}, we conclude 
	\begin{align*}
		\dim _{\mathrm{H} } \left ( \Lambda_{p} \left ( x \right )\right) \ge \dim _{\mathrm{H} } \left ( \Lambda_{p} \left ( x \right )\cap \left ( \beta_{n} ,\gamma_{n} \right )  \right ) \ge \dim _{\mathrm{H}}  \mathcal{F} _{k,n} \left ( x, p \right ) \ge \frac{\mathrm{h} \left ( p, 1-p \right ) }{ \log 3 }
	\end{align*}	
	as desired.
\end{proof}

\end{document}